\documentclass[12pt]{article}
\usepackage{amsfonts}
\usepackage{amsmath}
\usepackage{amssymb}
\usepackage{amsthm}
\usepackage{graphicx}
\graphicspath{{eps/}}
\usepackage{epstopdf}
\usepackage{epsfig}
\newtheorem{theorem}{Theorem}[section]
\newtheorem{lemma}{Lemma}[section]
\newtheorem{result}{Proposition}[section]
\newtheorem{corollary}{Corollary}[section]

\theoremstyle{definition}
\newtheorem{definition}{Definition}[section]
\newtheorem*{conjecture}{Conjecture}

\begin{document}

\begin{center}
\vskip 1cm{\LARGE\bf On Arithmetic Functions Related to Iterates of the Schemmel Totient Functions
\vskip 1cm
\large
Colin Defant\\
Department of Mathematics\\
University of Florida\\
United States\\
cdefant@ufl.edu}
\end{center}
\vskip .2 in

\begin{abstract}
We begin by introducing an interesting class of functions, known as the Schemmel totient functions, that generalizes the Euler totient function. For each Schemmel totient function $L_m$, we define two new functions, denoted $R_m$ and $H_m$, that arise from iterating $L_m$. Roughly speaking, $R_m$ counts the number of iterations of $L_m$ needed to reach either $0$ or $1$, and $H_m$ takes the value (either $0$ or $1$) that the iteration trajectory eventually reaches. Our first major result is a proof that, for any positive integer $m$, the function $H_m$ is completely multiplicative. We then introduce an iterate summatory function, denoted $D_m$, and define the terms $D_m$-deficient, $D_m$-perfect, and $D_m$-abundant. We proceed to prove several results related to these definitions, culminating in a proof that, for all positive even integers $m$, there are infinitely many $D_m$-abundant  numbers. Many open problems arise from the introduction of these functions and terms, and we mention a few of them, as well as some numerical results.
\end{abstract}

\section{Introduction}

Throughout this paper, $\mathbb{N}$, $\mathbb{N}_0$, and $\mathbb{P}$ will denote the set of positive integers, the set of nonnegative integers, and the set of prime numbers, respectively. For any function $f$, we will write $f^{(1)}=f$ and $f^{(k+1)}=f\circ f^{(k)}$ for all $k\in \mathbb{N}$.
The letter $p$ will always denote a prime number. For any $n\in \mathbb{N}$, $\upsilon_p(n)$ will denote the unique nonnegative integer $k$ such that $p^k\mid n$ and $p^{k+1}\nmid n$. Finally, in the canonical prime factorization $\prod_{i=1}^rp_i^{\alpha_i}$ of a positive integer, it is understood that, for all distinct $i,j\in \{1,2,\ldots,r\}$, we have $p_i\in \mathbb{P}$, $\alpha_i\in \mathbb{N}$, and $p_i\neq p_j$.

The well-known Euler $\phi$ function is defined to be the number of positive integers less than or equal to $n$ that are relatively prime to $n$. For each $m\in\mathbb{N}$, the Schemmel
totient function $L_m(n)$ is defined as the number of positive integers $k\leq n$ such that $\gcd(k+s, n)=1$ for all $s\in\{0,1,\ldots,m-1\}$ \cite{herzog92}. In particular, $L_1=\phi$. For no reason other than a desire to avoid cumbersome notation and the possibility of dealing with undefined objects such as $L_2^{(2)}(6)$, we will define $L_m(0)$ to be $0$ for all positive integers $m$.

For any integer $n>1$, let $p(n)$ be the smallest prime number that divides $n$. Schemmel  \cite{schemmel69} showed that, for any positive integer $m$, $L_m$ is multiplicative. Thus, $L_m(1)=1$. Furthermore, for $n>1$, \begin{equation}\label{1}L_m(n)=\begin{cases} 0, & \mbox{if } p(n)\leq m; \\ \displaystyle n\prod_{p\mid n}\left(1-\frac{m}{p}\right), & \mbox{if } p(n)>m. \end{cases}
\end{equation}
Letting $\prod_{i=1}^rp_i^{\alpha_i}$ be the canonical prime factorization of $n$, we may rewrite the above formula as
\begin{equation}\label{2}L_m(n)=\begin{cases} 0, & \mbox{if } p(n)\leq m; \\ \displaystyle \prod_{i=1}^rp_i^{\alpha_i-1}(p_i-m), & \mbox{if } p(n)>m \end{cases}
\end{equation}
for $n>1$.

In 1929, S. S. Pillai introduced a function that counts the number of iterations of the Euler $\phi$ function needed to reach $1$ \cite{pillai29}.  In the following section, we generalize Pillai's function via the Schemmel totient functions. Then, in the third section, we generalize the concept of perfect totient numbers with the introduction, for each positive integer $m$, of a function $D_m$, which sums the first $R_m$ iterates of $L_m$.

\section{The Functions $R_m$ and $H_m$}

We record the following propositions, which follow immediately from \eqref{2}, for later use.
\begin{result} \label{Res1.1}
 For $x,y,m\in \mathbb{N}$, if $x\vert y$, then $L_m(x)\vert L_m(y)$.
\end{result}
Repeatedly applying Proposition \ref{Res1.1}, we find
\begin{result} \label{Res1.2}
For $x,y,m,r\in \mathbb{N}$, if $x\vert y$, then $L_m^{(r)}(x)\vert L_m^{(r)}(y)$.
\end{result}
\begin{result} \label{Res1.3}
For $m,n\in \mathbb{N}$, if $m$ is even and $n$ is odd, then either $L_m(n)=0$ or $L_m(n)$ is odd.
\end{result}
In addition, the following theorem is now quite easy to prove.
\begin{theorem} \label{Thm1.1}
For any prime number $p$ and positive integer $x$,
\[L_{p-1}(px)=\begin{cases} L_{p-1}(x), & \mbox{if } p\nmid x; \\
pL_{p-1}(x), & \mbox{if } p \vert x. \end{cases}\]
\end{theorem}
\begin{proof}
If $p\nmid x$, it follows from the
 multiplicativity of $L_{p-1}$ that 
 $L_{p-1}(px)=L_{p-1}(p)L_{p-1}
(x)=L_{p-1}(x)$. If $p\vert x$, then we have 
\begin{align*} 
L_{p-1}(px)= L_{p-1}\left(p\cdot p^{\upsilon_p(x)}\cdot \frac{x}{p^{\upsilon_p(x)}}
\right)=L_{p-1}\left(p^{\upsilon_p(x)+1}\right)L_{p-1}\left(\frac{x}
{p^{\upsilon_p(x)}}\right) \\ = p^{\upsilon_p(x)}L_{p-1}\left(\frac{x}
{p^{\upsilon_p(x)}}\right)=pL_{p-1}\left(p^{\upsilon_p(x)}
\right)L_{p-1}\left(\frac{x}{p^{\upsilon_p(x)}}\right)=pL_{p-1}(x).
\end{align*}
\end{proof}
Notice that, for any positive integers $m$ and $n$ with $n>1$, we have $L_m(n)<n$ and $L_m(n)\in \mathbb{N}_0$. It is easy to see that, by starting with a positive integer $n$ and iterating the function $L_m$ a finite number of times, we must eventually reach either $0$ or $1$. More precisely, there exists a positive integer $k$ such that $L_m^{(k)}(n)\in \{0,1\}$. This leads us to the following definitions.
\begin{definition}
For all $m,n\in \mathbb{N}$, let $R_m(n)$ denote the least positive integer $k$ such that $L_m^{(k)}(n)\in \{0,1\}$. Furthermore, we define the function $H_m$ by \[H_m(n)=L_m^{(R_m(n))}(n).\]
\end{definition}
Though the functions $H_m$ only take values $0$ and $1$, they prove to be surprisingly interesting. For example, we can show that, for each positive integer $m$, $H_m$ is a completely multiplicative function. First, however, we will need some definitions and preliminary results.
\begin{definition}
For $m\in \mathbb{N}$, we define the following sets:
\[P_m=\{p\in\mathbb{P}: H_m(p)=1\}\]
\[Q_m=\{q\in\mathbb{P}: H_m(q)=0\}\]
\[S_m=\{n\in\mathbb{N}: q\nmid n\:\forall\: q\in Q_m\}\]
We define $T_m$ to be the unique set of positive integers defined by the following criteria:
\begin{itemize}
\item $1\in T_m$.

\item If $p$ is prime, then $p\in T_m$ if and only if $p-m\in T_m$.

\item If $x$ is composite, then $x\in T_m$ if and only if there exist $x_1,x_2\in T_m$ such that $x_1,x_2>1$ and $x_1x_2=x$.
\end{itemize}
\end{definition}
\begin{lemma} \label{Lem1.1}
Let $k,m\in\mathbb{N}$. If all the prime divisors of $k$ are in $T_m$, then all the positive divisors of $k$ (including $k$) are in $T_m$. Conversely, if $k\in T_m$, then every positive divisor of $k$ is an element of $T_m$.
\end{lemma}
\begin{proof}
First, suppose that all the prime divisors of $k$ are in $T_m$, and let $d$ be a positive divisor of $k$. Then all the prime divisors of $d$ are in $T_m$. Let $d=\prod_{i=1}^rw_i^{\alpha_i}$ be the canonical prime factorization of $d$. As $w_1\in T_m$, the third defining criterion of $T_m$ tells us that $w_1^2\in T_m$. Then, by the same token, $w_1^3\in T_m$. Eventually, we find that $w_1^{\alpha_1}\in T_m$. As $w_1^{\alpha_1},w_2\in T_m$, we have $w_1^{\alpha_1}w_2\in T_m$. Repeatedly using the third criterion, we can keep multiplying by primes until we find that $d\in T_m$. This completes the first part of the proof. Now we will prove that if $k\in T_m$, then every positive divisor of $k$ is an element of $T_m$. The proof is trivial if $k$ is prime, so suppose $k$ is composite. We will induct on $\Omega(k)$, the number of prime divisors (counting multiplicities) of $k$. If $\Omega(k)=2$, then, by the third defining criterion of $T_m$, the prime divisors of $k$ must be elements of $T_m$. Therefore, if $\Omega(k)=2$, we are done. Now, suppose the result holds whenever $\Omega(k)\leq h$, where $h>1$ is a positive integer. Consider the case in which $\Omega(k)=h+1$. By the third defining criterion of $T_m$, we can write $k=k_1k_2$, where $1<k_1,k_2<k$ and $k_1,k_2\in T_m$. By the induction hypothesis, all of the positive divisors of $k_1$ and all of the positive divisors of $k_2$ are in $T_m$. Therefore, all of the prime divisors of $k$ are in $T_m$. By the first part of the proof, we conclude that all of the positive divisors of $k$ are in $T_m$.
\end{proof}
\begin{theorem} \label{Thm1.2}
If $m$ is a positive integer, then $S_m=T_m$.
\end{theorem}
\begin{proof}
Fix $m\!\in\! \mathbb{N}$. Let $u$ be a positive integer such that,
for all $k\in \{1,2,\ldots,u-1\}$, either $k\in S_m$ and $k\in T_m$ or $k\not\in S_m$ and $k\not\in T_m$. We will show that $u\in S_m$ if and only
if $u\in T_m$. First, we must show that if $k\in \{1,2,\ldots,u-1\}$, then $k\in S_m$ if and only if $L_m(k)\in S_m$. Suppose, for the sake of
finding a contradiction, that $L_m(k)\in S_m$ and $k\not\in S_m$. As $k\not\in S_m$, we have that $k>1$ and $k\not\in T_m$. Lemma \ref{Lem1.1} then
guarantees that there exists a prime $q$ such that $q\vert k$ and $q\not\in T_m$. As $q\not\in T_m$, the second defining criterion of $T_m$
implies that $q-m\not\in T_m$. We know that $q>m$ because, otherwise, $p(k)\leq q\leq m$, implying that $L_m(k)=0\not\in S_m$. Therefore, $q-m\in \{1,2,\ldots,u-1\}$ and $q-m\not\in T_m$. By the induction hypothesis, $q-m\not\in S_m$. Therefore, there exists some $q_0\in Q_m$ such that $q_0\vert q-m$. Because $q\vert k$, Proposition \ref{Res1.1} implies that $L_m(q)\vert L_m(k)$. Thus, $q_0\vert q-m=L_m(q)\vert L_m(k)$, which implies
that $L_m(k)\not\in S_m$. This is a contradiction. Now suppose, so that
we may again search for a contradiction, that $L_m(k)\not\in S_m$ and $k\in S_m$. $L_m(k)\not\in S_m$ implies that $k>1$, and $k\in S_m$ implies
(by the induction hypothesis) that $k\in T_m$. By Lemma \ref{Lem1.1}, all positive divisors of $k$ are elements of $T_m$. Let $k=\prod_{i=1}^rp_i^{\alpha_i}$ be the canonical prime factorization of $k$. Then, by \eqref{2},
\[L_m(k)=\begin{cases} 0, & \mbox{if } p(k)\leq m; \\
\displaystyle \prod_{i=1}^rp_i^{\alpha_i-1}(p_i-m), & \mbox{if } p(k)>m. \end{cases}\]
If $p(k)\!\leq\! m$, then $H_m\left(p(k)\right)=0$, which mean that $p(k)\in Q_m$. As $p(k)\vert k$, we have contradicted $k\in S_m$. Therefore, $p(k)>m$, so $L_m(k)=\prod_{i=1}^rp_i^{\alpha_i-1}(p_i-m)$. For each $i\in \{1,2,\ldots,r\}$, $p_i$ is a positive divisor of $k$, so $p_i\in T_m$. The second criterion defining $T_m$ then implies that $p_i-m\in T_m$, so all positive divisors (and, specifically, all prime divisors) of $p_i-m$ are in $T_m$. This implies that all prime divisors of $L_m(k)$ are elements of $T_m$, so Lemma \ref{Lem1.1} guarantees that $L_m(k)\in T_m$. However, we have shown that $0<L_m(k)<k$, so $L_m(k)\in \{1,2,\ldots,u-1\}$. By the induction hypothesis, we have $L_m(k)\in S_m$, a contradiction. Thus, we have established that if $k\in \{1,2,\ldots,u-1\}$, then $k\in S_m$ if and only if $L_m(k)\in S_m$.

We are now ready to establish that $u\in S_m$ if and only if $u\in T_m$.  Suppose that $u\in S_m$ and $u\not\in T_m$. We know that $u>m$ because, otherwise, $L_m\left(p(u)\right)=H_m\left(p(u)\right)=0$, implying that $p(u)\in Q_m$ and contradicting $u\in S_m$. If $u$ is prime, then $u\in S_m$ implies that $u\in P_m$. Then $H_m(u)=H_m(u-m)=L_m^{\left(R_m(u-m)\right)}(u-m)=1\in S_m$. As $L_m^{\left(R_m(u-m)\right)}(u-m)=L_m\left(L_m^{\left(R_m(u-m)-1\right)}(u-m)\right)\in S_m$ (we assume here and in the rest of the proof that $R_m(u-m)$ is large enough so that the notation $L_m^{(\cdot)}$ makes sense as we have defined it, but the argument is valid in any case), it follows from the preceding argument that $L_m^{\left(R_m(u-m)-1\right)}(u-m)=L_m\left(L_m^{\left(R_m(u-m)-2\right)}(u-m)\right)\in S_m$. Continuing this pattern, we eventually find that $L_m(u-m)\in S_m$, so $u-m\in S_m$. By the induction hypothesis, $u-m\in T_m$. However, by the second criterion defining $T_m$, the primality of $u$ then implies that $u\in T_m$, a contradiction. Thus, $u$ must be composite. We assumed that $u\not\in T_m$, so Lemma \ref{Lem1.1} guarantees the existence of a prime $q\not\in T_m$ such that $q\vert u$. As $u$ is composite, $q\in \{1,2,\ldots,u-1\}$. The induction hypothesis then implies that $q\not\in S_m$, so $q\in Q_m$. However, this contradicts $u\in S_m$, so we have shows that if $u\in S_m$, then $u\in T_m$.
Suppose, on the other hand, that $u\not\in S_m$ and $u\in T_m$. Again, we begin by assuming $u$ is prime. Then, because $u\in T_m$, we must have $u-m\in T_m$. Therefore, by the induction hypothesis and the fact that $u-m\in \{1,2,\ldots,u-1\}$, it follows that $u-m\in S_m$. Now, $u\not\in S_m$, so we must have $u\in Q_m$. Therefore, $H_m(u)=H_m\left(L_m(u)\right)=H_m(u-m)=L_m^{\left(R_m(u-m)\right)}(u-m)=0\not\in S_m$. However, as $L_m^{\left(R_m(u-m)\right)}(u-m)=L_m\left(L_m^{\left(R_m(u-m)-1\right)}(u-m)\right)\not\in S_m$, it follows that $L_m^{\left(R_m (u-m)-1\right)}(u-m)=L_m\left(L_m^{\left(R_m(u-m)-2\right)}(u-m)\right)\not\in S_m$. Again, we continue this pattern until we eventually find that $L_m(u-m)\not\in S_m$, which means that $u-m\not\in S_m$. This is a contradiction, and we conclude that $u$ must be composite. From $u\in T_m$ and Lemma \ref{Lem1.1}, we conclude that all of the prime divisors of $u$ are elements of $T_m$. Furthermore, as $u$ is composite, all of the prime divisors of $u$ are elements of $\{1,2,\ldots,u-1\}$. Then, by the induction hypothesis, all of the prime divisors of $u$ are in the set $S_m$. This implies that none of the prime divisors of $u$ are in $Q_m$, so $u\in S_m$. This is a contradiction, and the induction step of the proof is finally complete. All that is left to check is the base  case. However, the base case is trivial because $1\in S_m$ and $1\in T_m$.
\end{proof}

We may now use the sets $S_m$ and $T_m$ interchangeably. In addition, part of the above proof gives rise to the following corollary.
\begin{corollary} \label{Corol}
Let $k,m,n\in \mathbb{N}$. Then $L_m^{(k)}(n)\in S_m$ if and only if $n\in S_m$.
\end{corollary}
\begin{proof}
The proof follows from the argument in the above proof that $L_m(n)\in S_m$ if and only if $n\in S_m$ whenever $n\in \{1,2,\ldots,u-1\}$. As we now know that we can make $u$ as large as we need, it follows that $L_m(n)\in S_m$ if and only if $n\in S_m$. Then $L_m^{(2)}(n)\in S_m$ if and only if $L_m(n)\in S_m$, $L_m^{(3)}(n)\in S_m$ if and only if $L_m^{(2)}(n)\in S_m$, and, in general, $L_m^{(r+1)}(n)\in S_m$ if and only if $L_m^{(r)}(n)\in S_m$ $(r\in \mathbb{N})$. The desired result follows immediately.
\end{proof}
\begin{corollary} \label{Cor1.1}
Let $m,n\in\mathbb{N}$. Then $H_m(n)\in S_m$ if and only if $n\in S_m$.
\end{corollary}
\begin{proof}
It is clear that $H_m(n)\in S_m$ if and only if $H_m(n)=1$. Therefore, the proof follows immediately from setting $k=R_m(n)$ in Corollary \ref{Corol}.
\end{proof}

Notice that, for a given positive integer $m$, Corollary \ref{Cor1.1}, along with Theorem \ref{Thm1.2} and the defining criteria of $T_m$, provides a simple way to construct the set of all positive integers $x$ that satisfy $H_m(x)=1$. Corollary \ref{Cor1.1} also expedites the proof of the following theorem.
\begin{theorem} \label{Thm1.3}
The function $n\mapsto H_m(n)$ is completely multiplicative for all $m\in\mathbb{N}$. 
\end{theorem}
\begin{proof}
Choose some $m,x,y\in\mathbb{N}$. First, suppose $H_m(x)=0$. By Corollary \ref{Cor1.1}, $x\not\in S_m$. Therefore, there exists $q\in Q_m$ such that $q\vert x$. This implies that $q\vert xy$, so $xy\not\in S_m$. Thus, $H_m(xy)=0$. A similar argument shows that $H_m(xy)=0$ if $H_m(y)=0$. Now, suppose that $H_m(x)=H_m(y)=1$. Then Corollary \ref{Cor1.1} ensures that $x,y\in S_m$. Therefore, $xy\in S_m$, so $H_m(xy)=1$. As the function $H_m$ can only take values $0$ and $1$, the proof is complete.
\end{proof}

In concluding this section, we note that if $m+1$ is composite, then it is impossible for any integer greater than $1$ to be in $S_m$. Therefore, whenever $m+1$ is composite, we have $H_m(1)=1$ and $H_m(n)=0$ for all integers $n>1$.
\section{Summing the Iterates}
A perfect totient number is defined \cite{iannucci03} to be a positive integer $n>1$ that satisfies (using our previous notation) \[n=\sum_{i=1}^{R_1(n)}\phi^{(i)}(n).\]
In the following definitions, we generalize the concept of perfect totient numbers. We also borrow some other traditional terminology related to perfect numbers.
\begin{definition}
Let $m$ be a positive integer. We define the arithmetic function $D_m$ by $D_m(1)=0$ and \[D_m(n)=\sum_{i=1}^{R_m(n)}L_m^{(i)}(n)\] for all integers $n>1$. If $D_m (n)<n$, we say that $n$ is $D_m$\textit{-deficient}. If $D_m(n)=n$, we say that $n$ is $D_m$\textit{-perfect}. If $D_m(n)>n$, we say that $n$ is $D_m$\textit{-abundant}. Finally, in the case when $D_m(n)=0$, we say that $n$ is $D_m$\textit{-stagnant}.
\end{definition}

We now present a series of theorems related to these definitions.
\begin{theorem} \label{Thm2.1}
If $m>1$ is odd, then all positive integers are $D_m$-deficient.
\end{theorem}

\begin{proof}
Let $m>1$ be an odd integer, and let $n$ be any positive integer. If $n=1$ or $p(n)\leq m$, then $n$ is $D_m$-stagnant. \textit{A fortiori}, $n$ is $D_m$-deficient. If $p(n)>m$, then $p(n)-m$  is even and $p(n)-m|L_m (n)$ (by Proposition \ref{Res1.1}). Thus, $2\vert L_m(n)$, which implies that $L_m^{(2)}(n)=0$. Hence, $D_m(n)=L_m(n)<n$.

\end{proof}
\begin{theorem} \label{Thm2.2}
All positive even integers are $D_m$-deficient for all positive integers $m$.
\end{theorem}

\begin{proof}
The proof is trivial for $m>1$ because, in that case, any positive even integer is clearly $D_m$-stagnant. For $m=1$, we use the fact that all totient numbers greater than $1$ are even. Therefore, \[D_1(n)=\phi(n)+\phi^{(2)}(n)+\cdots+\phi^{\left(R_1(n)\right)}(n)\leq \frac{1}{2}n+\frac{1}{4}n+\cdots+\frac{1}{2^{R_1(n)}}n<n.\]
\end{proof}
Theorem \ref{Thm2.2} is nothing revolutionary, but we include it because it fits nicely with the next theorems.

For the next two theorems, which are not quite as trivial as the previous two, we require the following lemma.
\begin{lemma} \label{Lem2.1}
If $k>1$ is an odd integer, then at least one element of the set $\{k,L_2(k),L_2^{(2)}(k)\}$ is divisible by $3$.
\end{lemma}
\begin{proof}
Let $k>1$ be an odd integer with prime factor $p$, and suppose $3\nmid k$ and $3\nmid L_2(k)$. We know that $p\not\equiv 2$ (mod $3$) because $p-2\vert L_2(k)$, so $p\equiv 1$ (mod $3$). As $p-2\equiv 2$ (mod $3$), $p-2$ must have some prime factor $p'$ such that $p'\equiv 2$ (mod $3$). But then, using Proposition \ref{Res1.2}, $3\vert p'-2=L_2(p')\vert L_2(p-2)=L_2^{(2)}(p)\vert L_2^{(2)}(k)$.
\end{proof}
\begin{theorem} \label{Thm2.3}
For any integer $m>1$, all positive multiples of $3$ are $D_m$-deficient.
\end{theorem}
\begin{proof}
If $m\geq 3$, then any positive multiple of $3$ is clearly $D_m$-stagnant. Therefore, we only need to check the case $m=2$. Write $K=\{n\in \mathbb{N}:3\vert n, D_2(n)\geq n\}$. Suppose $K\not=\emptyset$ and let $n_0$ be the smallest element of $K$. If $n_0=3^{\alpha}$ for some $\alpha\in\mathbb{N}$, then $D_2(n_0)=3^{\alpha-1}+3^{\alpha-2}+\cdots+3+1=\displaystyle{\frac{n_0-1}{2}}<n_0$. Therefore, $n_0$ must have some prime divisor $p\not=3$. From Theorem \ref{Thm2.1}, $p\not=2$. Also, by Proposition \ref{Res1.2}, $L_2(p)\vert L_2(n_0)$ and $L_2^{(2)}(p)\vert L_2^{(2)}(n_0)$. By Lemma \ref{Lem2.1}, at least one of $L_2(p)$ and $L_2^{(2)}(p)$ must be divisible by $3$. Suppose $3\vert L_2(p)$ so that $3\vert L_2(n_0)$. By the choice of $n_0$ as the smallest element of $K$, we have $D_2\left(L_2(n_0)\right)<L_2(n_0)$. This implies that $D_2(n_0)=L_2(n_0)+D_2\left(L_2(n_0)\right)<2L_2(n_0)<\displaystyle{\frac{2}{3}}n_0$ because $3\vert n_0$. From this contradiction, we conclude that $3\vert L_2^{(2)}(p)$, so $3\vert L_2^{(2)}(n_0)$. Again, by the choice of $n_0$, we have
$D_2\left(L_2^{(2)}(n_0)\right)<L_2^{(2)}(n_0)$. However, this implies that $D_2(n_0)=L_2(n_0)+L_2^{(2)}(n_0)+D_2\left(L_2^{(2)}(n_0 )\right)<L_2(n_0)+2L_2^{(2)}(n_0)<3L_2(n_0)<n_0$, which is a contradiction. It follows that $K$ is empty.
\end{proof}
\begin{theorem} \label{Thm2.4}
If $m>1$ is a positive integer and $m\not=4$, then all positive multiples of $5$ are $D_m$-deficient.
\end{theorem}
\begin{proof}
Let $m$ be a positive integer other than $1$ or $4$, and let $n$ be a multiple of $5$. If $m\geq 5$, then $n$ is $D_m$-stagnant. If $m=3$, then $n$ is $D_m$-deficient by Theorem \ref{Thm2.1}. We therefore only need to check the case $m=2$. Write $n=5^{\alpha}k$, where $\alpha,k\in \mathbb{N}$. We may assume that $2\nmid k$ and $3\nmid k$ because, otherwise, the desired result follows immediately from either Theorem \ref{Thm2.2} or Theorem \ref{Thm2.3}. We now consider two cases.

Case 1: $\alpha\geq 2$. Write $L_2(k)=3^{\alpha_1}5^{\alpha_2}t$, where $t$ is a positive integer not divisible by $2$, $3$, or $5$ and $\alpha_1,\alpha_2\in \mathbb{N}_0$ (we use Proposition \ref{Res1.3} to conclude that $t$ is odd). Then $L_2(n)=L_2(5^{\alpha})L_2(k)=3^{\alpha_1+1}5^{\alpha+\alpha_2-1}t$
and $L_2^{(2)}(n)=L_2(3^{\alpha_1+1})L_2(5^{\alpha+\alpha_2-1})L_2(t)=3^{\alpha_1+1}5^{\alpha+\alpha_2-2}L_2(t)$. As $3\vert L_2(n)$,
we can use Theorem \ref{Thm2.3} to write $D_2(n)=L_2(n)+L_2^{(2)}(n)+D_2\left(L_2^{(2)}(n)\right)<L_2(n)+2L_2^{(2)}(n)=3^{\alpha_1+1}5^{\alpha+\alpha_2-1}t+2\left(3^{\alpha_1+1}5^{\alpha+\alpha_2-2}L_2(t)\right)\leq 7\left(3^{\alpha_1+1}5^{\alpha+\alpha_2-2}t\right)\leq \displaystyle{\frac{21}{25}5^{\alpha}k=\frac{21}{25}n}$. This completes the proof of the case when $\alpha\geq 2$.

Case 2: $\alpha=1$. In this case, $n=5k$, so $L_2(n)=3L_2(k)$. We may assume that $k>1$ because the case $n=5$ is trivial. First, suppose that $3\vert L_2(k)$. In this case, $L_2^{(2)}(k)\leq \displaystyle{\frac{1}{3}L_2(k)}$, and, by Theorem \ref{Thm1.1}, $L_2^{(2)}(n)=3L_2^{(2)}(k)$. Then, using Theorem \ref{Thm2.3}, we have $D_2(n)=L_2(n)+L_2^{(2)}(n)+D_2\left(L_2^{(2)}(n)\right)=3L_2(k)+3L_2^{(2)}(k)+D_2\left(3L_2^{(2) }(k)\right)<3L_2(k)+6L_2^{(2)}(k)\leq 5L_2(k)\leq n-10$. Now suppose that $3\nmid L_2(k)$. By Lemma \ref{Lem2.1} and our assumption that $3\nmid k$, we have $3\vert L_2^{(2)}(k)$. Using Theorem \ref{Thm1.1} and Theorem \ref{Thm2.3} again, we have $D_2(n)=L_2(n)+L_2^{(2)}(n)+D_2\left(L_2^{(2)}(n) \right)=3L_2(k)+L_2^{(2)}(k)+D_2\left(L_2^{(2)}(k)\right)<3L_2(k)+2L_2^{(2)}(k)\leq 3(k-2)+2(k-4)=n-14$. This completes the proof of all cases.
\end{proof}

The last few theorems have dealt with $D_m$-deficient numbers, so it is natural to ask questions about $D_m$-abundant numbers. We might wish to know the positive integers $m$ for which $D_m$-abundant numbers even exist. How many $D_m$-abundant numbers exist for a given $m$? How large can we make $D_m(n)-n$? Theorem \ref{Thm2.1} deals with these questions for the cases when $m$ is odd and greater than $1$. Also, a great deal of literature \cite{iannucci03,luca06} already exists concerning the case $m=1$. In the following theorem, we answer all of the preceding questions for the cases when $m$ is a positive even integer.
\begin{theorem} \label{Thm2.5}
Let $m$ be a positive even integer. For any positive $A$ and $\delta$, there exist infinitely many primes $p_0$ such that $L_m(p_0^{\alpha})+L_m^{(2)}(p_0^{\alpha})>p_0^{\alpha}+Ap_0^{\alpha-\delta}$ for all positive integers $\alpha$.
\end{theorem}
\begin{proof}
Fix $m$, $A$, and $\delta$ to be positive real numbers, where $m$ is an even integer. Let $p_1,p_2,\ldots,p_r$ be all the primes that divide $m$, and let $q_1,q_2,\ldots,q_t$ be all the primes that are less than $m$ and do not divide $m$. For each $j\in \{1,2,\ldots,t\}$, define $\sigma_j$ by \[\sigma_j=\begin{cases} 1, & \mbox{if } m\not\equiv 1 \pmod{q_j}; \\
-1, & \mbox{if } m\equiv 1 \pmod{q_j}. \end{cases}\]
Write $\displaystyle M=\prod_{p \leq m}p$. By the Chinese remainder theorem, there exists a unique solution modulo $M$ to the system of congruences defined by
\begin{equation}\label{4}
\begin{cases} x\equiv 1 \pmod{p_i} & \mbox{if } i\in \{1,2,\ldots,r\}; \\
x \equiv \sigma_j \pmod{q_j}  & \mbox{if } j\in \{1,2,\ldots,t\}. \end{cases}
\end{equation}
It is easy to see that if $x_0$ is a solution to \eqref{4}, then $x_0$ and $x_0-m$ are each relatively prime to every prime less than or equal to $m$. By Dirichlet's theorem concerning the infinitude of primes in arithmetic progressions, there must be infinitely many primes that satisfy the system \eqref{4}. Let $p_0$ be one such prime, and write \[\beta=\prod_{p\vert p_0-m}\left(1-\frac{m}{p}\right).\]
As $p_0$ is relatively prime to all primes less than or equal to $m$, we have \[\prod_{m<p\leq p_0}\left(1-\frac{m}{p}\right)\leq \beta \leq 1.\]
It is well-known \cite{rosser62}, that, as $p_0\rightarrow\infty$, \[\prod_{m<p\leq p_0}\left(1-\frac{m}{p}\right)\sim\frac{c_m}{(\log p_0)^m}\] for some constant $c_m$ that depends only on $m$. We find that
\[\beta p_0^2\geq p_0^2\prod_{m<p\leq p_0}\left(1-\frac{m}{p}\right)\sim\frac{c_mp_0^2}{(\log p_0)^m}.\]
Therefore, we may choose $p_0$ to be large enough so that $\beta p_0^2>Ap_0^{2-\delta}+3mp_0$. With this choice of $p_0$, we may write $\beta(p_0-m)^2>\beta p_0^2-2\beta mp_0\geq \beta p_0^2-2mp_0>Ap_0^{2-\delta}+mp_0$. But $\beta (p_0-m)=L_m(p_0-m)$ because $p(p_0-m)>m$. Thus,
\begin{equation}\label{5}
(p_0-m)L_m(p_0-m)>Ap_0^{2-\delta}+mp_0.
\end{equation}
Let $\alpha$ be an integer, and, for now, assume $\alpha\geq 2$. Rearranging and multiplying the inequality \eqref{5} by $p_0^{\alpha-2}$, we have $-mp_0^{\alpha-1}+p_0^{\alpha-1}L_m(p_0-m)>mp_0^{\alpha-2}L_m(p_0-m)+Ap_0^{\alpha-\delta}$. After further algebraic manipulation, we find $p_0^{\alpha-1}(p_0-m)+p_0^{\alpha-2}(p_0-m)L_m(p_0-m)>p_0^{\alpha}+Ap_0^{\alpha-\delta}$. Noticing that the left-hand side of the preceding inequality is simply $L_m(p_0^{\alpha})+L_m^{(2)}(p_0^{\alpha})$, we have $L_m(p_0^{\alpha})+L_m^{(2)}(p_0^{\alpha})>p_0^{\alpha}+Ap_0^{\alpha-\delta}$. This is the desired result for $\alpha\geq 2$. To show that the result holds when $\alpha=1$, it suffices to show that $L_m(p_0)+L_m^{(2)}(p_0)>\displaystyle{\frac{L_m(p_0^2)+L_m^{(2)}(p_0^2)}{p_0}}$. This reduces to $p_0-m+L_m(p_0-m)>p_0-m+\displaystyle{\frac{p_0-m}{p_0}L_m(p_0-m)}$, which is obviously true.
\end{proof}
\begin{corollary} \label{Cor2.1}
For any positive even integer $m$, there exist infinitely many $D_m$-abundant numbers.
\end{corollary}

We conclude this section with a remark about $D_m$-perfect numbers. Using \textit{Mathematica}, one may check that for $m\in \{2,4,6\}$, the only $D_m$-perfect number less than $100,000$ is $37,147$, which is $D_2$-perfect. Unfortunately, this data is too scarce to make any reasonable conjecture about the nature or distribution of $D_m$-perfect numbers for positive even integers $m$.
\section{Numerical Analysis and\\ Concluding Remarks}

In 1943, H. Shapiro investigated a function $C$, which counts the number of iterations of the $\phi$ function needed to reach $2$ \cite{shapiro43}. Shapiro showed that the function $C$ is additive, and he established bounds for its values. In this paper, we have not gone into much detail exploring the functions $R_m$ because they prove, in general, to be either completely uninteresting or very difficult to handle. For example, for any integer $n>1$,
\[R_3(n)=\begin{cases} 1, & \mbox{if } n \not\equiv 1,5 \pmod{6}; \\
2, & \mbox{if } n\equiv 1,5 \pmod{6}. \end{cases}\]
On the other hand, the function $R_4$ does not seem to obey any nice pattern or exhibit any sort of nice additive behavior. There seems to be some hope in analyzing the function $R_2$, so we make the following conjecture.
\begin{conjecture}
If $x>3$ is an odd integer, then
\[R_2(x)\geq \frac{\log \left(\frac{49}{15}x\right)}{\log 7}.\]
\end{conjecture}

We note that it is not difficult to prove, using Lemma \ref{Lem2.1} and a bit of case work, that $R_2(x)\leq 3\displaystyle{\frac{\log(x+2)}{\log3} -3}$ for all integers $x>1$ (with equality only at $x=7$). However, as Figure 1 shows, this is a very weak upper bound (at least for relatively small $x$). It is tempting to think, based on the figure, that $R_2(x)\leq 3+\displaystyle{\frac{\log x}{\log 3}}$ for all positive integers $x$. However, setting $x=480,314,203$ yields a counterexample because $3+\displaystyle{\frac{\log 480,314,203}{\log 3}\approx 21.196<22=R_2(480,314,203)}$.

\begin{figure}
\epsfysize=7cm \epsfbox{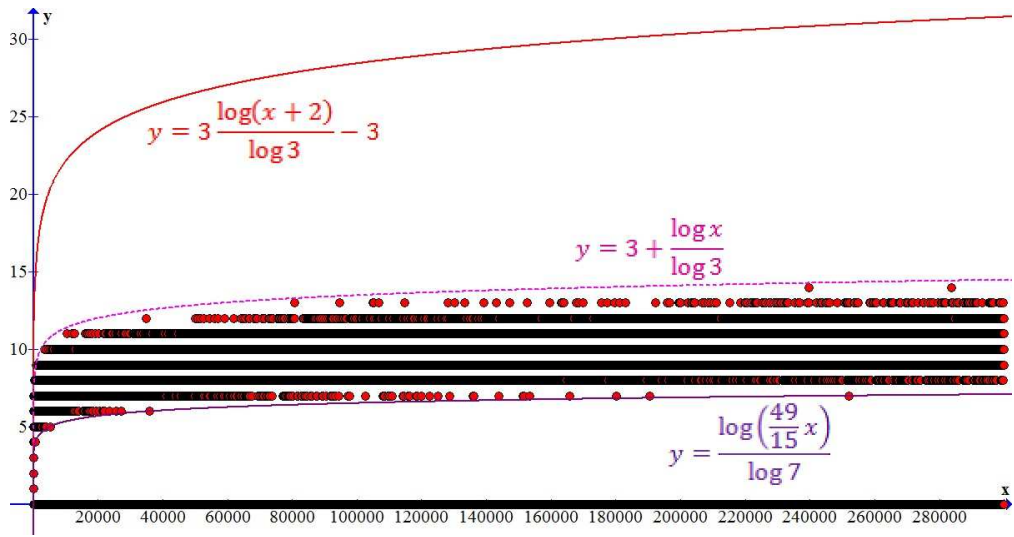}
\caption{A plot of the first $300,000$ values of the function $R_2$, as well as some important curves. Note that the black streaks in the figure are, in actuality, several overlapping dots.}
\end{figure}

The author has found that investigating bounds of the function $R_2$ naturally leads to a question about the infinitude of twin primes, which hints at the potential difficulty of the problem. Indeed, Harrington and Jones \cite{harrington10} have arrived at the same conclusion while studying the function $C_2(x):=R_2(x)-1$, and they conjecture that the values of $C_2(x)+C_2(y)-C_2(xy)$ can be arbitrarily large. To avoid the unpredictability of the values of the function $C_2$, Harrington and Jones have restricted the domain of $C_2$ to the set $D$ of positive integers $k$ with the property that none of the numbers in the set $\{k,L_2(k),L_2^{(2)}(k),\ldots\}$ has a prime factor that is congruent to $1$ modulo $3$. With this restriction of the domain of $C_2$, these two authors have established results analogous to those that Shapiro gave for the function $C$ mentioned earlier. In fact, we speculate that methods analogous to those that Harrington and Jones have used could easily generalize to allow for analogous results concerning functions $C_m(x):=R_m(x)-1$ if one is willing to use a sufficiently restricted domain of $C_m$.  

We next remark that, in Theorem \ref{Thm2.4}, the requirement that $m\not=4$ is essential. For example, write $p_1=306,167$, $p_2=4+p_1^2$, $p_3=4+p_2^2$, $p_4=4+p_3^2$, and $p_5=4+p_4^2$. Then the number $5p_5$ is a $D_4$-abundant multiple of 5.

Lastly, we have not spent much effort analyzing the ``sizes" of the functions $D_m$ or searching for $D_m$-perfect numbers. We might inquire about the average order or possible upper and lower bounds for $D_m$ for a general positive even integer $m$. In addition, it is natural to ask if there even are any $D_m$-perfect numbers other than $37,147$ for even positive integers $m$.

\section{Acknowledgments and Dedications}
Dedicated to my parents Marc and Susan, my brother Jack, and my sister Juliette.

Also dedicated to Mr.~Jacob Ford, who wrote a program to find values of the functions $L_m$, $R_m$, and $H_m$. Mr.~Ford and my father also sparked my interest in computer programming, which I used to analyze the functions $D_m$.

Finally, I would like to thank the unknown referee for taking the time to read carefully through my work and for his or her valuable suggestions.

\bigskip
\hrule
\bigskip

\noindent 2010 {\it Mathematics Subject Classification}:  Primary 11N64; Secondary 11B83.

\noindent \emph{Keywords: } Schemmel totient function, iterated arithmetic function, summatory function, perfect totient number.

\bigskip
\hrule
\bigskip

\noindent (Concerned with sequences A000010, A003434, A058026, A092693, A123565, A241663, A241664, A241665, A241666, A241667, and A241668.)
\end{document}